\documentclass{amsart}
\usepackage[utf8]{inputenc}
\usepackage{a4wide}
\usepackage{amsmath,amsfonts,amssymb}
\usepackage{amsthm}
\usepackage{ifthen}
\usepackage{longtable}
\usepackage{array}
\usepackage{url}
\usepackage{hyperref}
\usepackage{enumerate}
\usepackage{enumitem}

\newcommand{\Z}{\mathbb{Z}}
\newcommand{\Q}{\mathbb{Q}}

\newcommand{\K}{\mathbb{K}}

\renewcommand{\phi}{\varphi}

\newcommand{\Norm}{\mathcal{N}}

\newcommand{\eps}{\varepsilon}
\newcommand{\bb}{,\ldots ,}

\usepackage{titlesec}
\titleformat{\section}
  {\center\normalsize\bfseries}{\thesection.}{1em}{}

\newtheorem*{theorem*}{Theorem}
\newtheorem{thm}{Theorem}
\newtheorem{lem}{Lemma}

\newtheorem*{claim*}{Claim}

	\newcounter{countknownthm}

\theoremstyle{definition}

	\newtheorem{rem}{Remark}

\makeatletter
\@namedef{subjclassname@2020}{%
  \textup{2020} Mathematics Subject Classification}
\makeatother

\begin{document}
\title[Family of Thue Equations involving Fibonacci Numbers and Powers of Two]{On a cubic Family of Thue Equations involving Fibonacci Numbers and Powers of Two}
\thanks{The author was supported by the Austrian Science Fund (FWF) under the project I4406.
}

\author[I. Vukusic]{Ingrid Vukusic}
\address{I. Vukusic,
University of Salzburg,
Hellbrunnerstrasse 34/I,
A-5020 Salzburg, Austria}
\email{ingrid.vukusic\char'100sbg.ac.at}

\begin{abstract}
In this paper we completely solve the family of parametrised Thue equations
\[
	X(X-F_n Y)(X-2^n Y)-Y^3=\pm 1,
\]
where $F_n$ is the $n$-th Fibonacci number. In particular, for any integer $n\geq 3$ the Thue equation has only the trivial solutions $(\pm 1,0), (0,\mp 1), \mp(F_n,1), \mp(2^n,1)$.
\end{abstract}

\maketitle

\section{Introduction}

Thue equations are Diophantine equations of the form
\[
	F(X,Y)=m,
\]
where $F \in \Z[X,Y]$ is an irreducible homogeneous polynomial of degree $d\geq 3$ and $m$ is a fixed non-zero integer. In 1909 A. Thue \cite{Thue1909} proved that such an equation has at most finitely many solutions. Thanks to Baker's theory of linear forms in logarithms and many other contributions, it is today possible to (at least in principle) solve any Thue equation effectively and in many cases also rather efficiently. For an overview of the methods and references see e.g. \cite[Chapter VII]{Smart1998}.

Less fully understood are parametrised families of Thue equations, that is Thue equations where the coefficients in $F$ depend on one or more parameters. Already Thue himself \cite{Thue1918} considered a family of Thue equations, namely $aX^d - bY^d =c$, and hinted that in particular the family
\[
	(n+1)X^d - nX^d=1
\]
might be solved for $d$ prime. In the meantime this family has been solved completely for arbitrary degree $d\geq 3$ by Bennett \cite{Bennett2001}.

E. Thomas investigated other types of parametrised Thue equations. In 1990 he  \cite{Thomas1990} completely solved the cubic family
\[
	X^3 - (n-1) X^2 Y - (n+2) XY^2 - Y^3 =1.
\] 
Since then, many families of Thue equations have been studied by various authors, for a detailed list see \cite{Heuberger2006}. In many of the families the coefficients of $F$ are specific polynomials in $n$. There are also more general results: Thomas \cite{Thomas1993} considered families of the form
\begin{equation}\label{eq:thomas}
	\prod_{i=1}^d (X-p_i(n)Y)- Y^d = \pm 1,
\end{equation}
where the $p_i(n)$ are polynomials. He conjectured that if the polynomials are monic, $p_1=0$ and $\deg p_2 < \dots < \deg p_n$, then for sufficiently large $n$ the only solutions to \eqref{eq:thomas} are $(x,y)=(\pm 1,0)$ and $(x,y)=\mp (p_i,1)$ for $i=1\bb n$. He proved his conjecture in the case $d=3$ under some technical hypothesis. The original conjecture, however, is wrong, as Ziegler \cite{Ziegler2007} provided two counterexamples. It is not yet clear if these are the only two counterexamples in the case $d=3$ and whether there are counterexamples for $d\geq 4$. 
Heuberger \cite{Heuberger2001} generalised Thomas' result and proved his conjecture under some technical hypothesis for any $d\geq 3$. 

Recently, the authors in \cite{HilgartVukusicZiegler2021} replaced the polynomials in \eqref{eq:thomas} by some specific linear recurrence sequences.
They solved the family
\[
	X(X-F_nY)(X-L_nY)-Y^3=\pm 1,
\]
where $F_n$ is the $n$-th Fibonacci number and $L_n$ is the $n$-th Lucas number. To the author's knowledge, this was the first time that an exponentially parametrised family was considered.

In this paper, we solve a similar cubic family of Thue equations, but in our case the two occurring linear recurrence sequences shall have dominant roots of distinct size.
This is comparable to the assumption in Thomas' conjecture that the polynomials should have increasing degrees.

To be precise, we will consider powers of two and the Fibonacci numbers, defined by $F_0=0,F_1=1$ and $F_{n+1}=F_n+F_{n-1}$ for $n\geq 1$. We will prove

\begin{thm}\label{thm:main}
For any integer $n\geq 3$ the Thue equation
\begin{equation}\label{eq:main}
	X(X-F_n Y)(X-2^n Y)-Y^3=\pm 1
\end{equation}
has only the trivial solutions
\begin{equation}\label{eq:trivial_sols}
(\pm 1,0), (0,\mp 1), \mp(F_n,1), \mp(2^n,1).
\end{equation}
For $n=1$ and $n=2$ the only non-trivial solutions are $\pm(7,3)$ and $\mp(1,2)$ respectively.
\end{thm}

We will use Heuberger's method \cite{Heuberger2001} as well as standard methods for solving Thue equations. 

In the next section we will recall some auxiliary results, including a lower bound for linear forms in logarithms and two reduction methods. In Section~\ref{sec:setup} we first solve Equation~\eqref{eq:main} for $n\leq 28$ with a computer. Then, assuming $n>28$, we do the standard preparations for solving such Thue equations, like estimating the roots of the polynomial $X(X-F_n)(X-2^n)-1$. In Section~\ref{sec:uBy-n1000} we find an upper bound for $\log |y|$, which is cubic in $n$. Moreover, we solve Equation \eqref{eq:main} for all $n\leq 1000$. Assuming $n>1000$, in Section \ref{sec:lowerBoundLogY} we find a lower bound for $\log |y|$ which is exponential in $n$. Thus, from the upper and the lower bound for $\log |y|$, we get an absolute upper bound for $n$ in Section \ref{sec:finish}. We finish the proof by reducing $n$ to less than 1000 with an LLL reduction. 

Finally, let us mention that the type $j$ of a solution plays an important role in the proof. If $(x,y)$ is a solution to \eqref{eq:main}, then by the general theory of Thue equations $x/y$ will be an extremely good approximation to one of the roots of $X(X-F_n)(X-2^n)-1=(X-\alpha^{(1)})(X-\alpha^{(2)})(X-\alpha^{(3)})$. If $x/y$ approximates $\alpha^{(j)}$, then we say that $(x,y)$ is of type $j$. In the proof it will make a difference whether $j=1,2$ or $j=3$ (which is the case when $x/y$ approximates the largest root). 
In fact, in Section \ref{sec:lowerBoundLogY} it will be much easier to obtain a good lower bound for $\log |y|$ in the case $j=3$.
Also in \cite{Heuberger2001} the cases $j=1,2$ are the particularly difficult ones. Moreover, because we exchange polynomials for recurrence sequences, we need to apply linear forms in logarithms one more time than in \cite{Heuberger2001} in these cases.
On the other hand, in Section \ref{sec:uBy-n1000} the case $j=3$ will require more computations than $j=1,2$.

\section{Auxiliary results}\label{sec:aux}

In this section we recall the Binet formula for the Fibonacci numbers, the notion of heights, lower bounds for linear forms in logarithms, the Baker-Davenport reduction method, the reduction method using the LLL-algorithm, and some elementary inequalities.

For the Fibonacci numbers we have the well known Binet formula
\[
	F_n = \frac{\phi^n - \psi^n}{\sqrt{5}},
	\quad \text{where} \quad
	\phi = \frac{1+\sqrt{5}}{2}
	\quad \text{and} \quad
	\psi = \frac{1-\sqrt{5}}{2}.
\]

Now we recall the notion of heights.
Let $\gamma$ be an algebraic number of degree $d\geq 1$ with the minimal polynomial
\[
	a_d X^d + \dots +a_1 X + a_0 
	= a_d \prod _{i=1}^{d} (X- \gamma^{(i)}),
\]
where $a_0, \dots, a_d$ are relatively prime integers and $\gamma^{(1)}, \dots, \gamma^{(d)}$ are the conjugates of $\gamma$. Then the \textit{logarithmic height} of $\gamma$ is given by
\[
	h(\gamma)
	= \frac{1}{d} \left(
		\log |a_d|
		+ \sum_{i=1}^d \log \left( \max \{ 1,|\gamma^{(i)} |\} \right)
		\right).
\]
The following well-known properties of the logarithmic height hold for any algebraic numbers $\gamma_1\bb \gamma_t$ and $s\in \Z$:
\begin{itemize}
\item $h(\gamma_1\cdots \gamma_t) \leq h(\gamma_1)+\dots + h(\gamma_t)$,
\item $h(\gamma_1 + \dots + \gamma_t)\leq h(\gamma_1) + \dots + h(\gamma_t) +\log t$,
\item $h(\gamma^s)=|s|h(\gamma)$.
\end{itemize}

Next, we state a lower bound for linear forms in three logarithms, which is a special case of \cite[Corollary 2.3]{Matveev2000}.

\begin{lem}[Matveev]\label{lem:matveev}
 Let $\gamma_1, \gamma_2, \gamma_3$ be non-zero real algebraic numbers in a number field $\K$ of degree $D$, let $b_1,b_2,b_3$ be rational integers, and let
\[
	\Lambda:= b_1 \log \gamma_1 + b_2 \log \gamma_2 + b_3 \log \gamma_3
\]
be non-zero. Then 
\[
	\log |\Lambda|>
		-1.4\cdot 10^{11} D^2 \log (eD) \log (eB) A_1 A_2 A_3,
\]
where
\[
	B\geq \max \left\{|b_1|,|b_2|,|b_3|\right\}
\]
and
\[
	A_i\geq \max\left\{D h(\gamma_i),|\log \gamma_i|,0.16 \right\}
	\qquad (i=1,2,3).
\]
\end{lem}

With Lemma \ref{lem:matveev} we will obtain huge bounds and we will reduce them using the Baker-Davenport reduction method and the LLL-algorithm.

Let us denote the distance to the nearest integer by $\|\cdot \|$. We phrase the Baker-Davenport reduction method in terms of logarithms because we will apply it in this form. The following lemma is an immediate variation of \cite[Lemma 5]{DujellaPetho1998}.

\begin{lem}[Baker-Davenport reduction] \label{lem:Baker-Davenport} 
Let $\gamma_0, \gamma_1 , \gamma_2, c, k$ be positive real numbers and let 
$u_1, u_2$ be integers bounded in absolute values by a positive number $M$.
Assume that
\[
	|\log \gamma_0 + u_1 \log \gamma_1 + u_2 \log \gamma_2|
	< c \cdot w^{-k}
\]
is satisfied for a positive number $w$.
Let $\frac{p}{q}$ be a convergent of the continued fraction of $\frac{\log \gamma_1}{\log \gamma_2}$. If
\[
	\eps :=\left\| \frac{\log \gamma_0}{\log \gamma_2} q \right\| 
		- M \left\| \frac{\log \gamma_1}{\log \gamma_2} q \right\|
\] 
is positive, then it follows that
\[
	w < \left( \frac{qc}{\eps |\log \gamma_2|} \right)^{1/k}.
\]
\end{lem}

The next lemma describes the way in which we will apply the LLL-algorithm to find good lower bounds for linear forms in logarithms. It is an immediate variation of \cite[Lemma VI.1]{Smart1998}.

\begin{lem}[LLL reduction]\label{lem:LLL}
Let $\gamma_1,\gamma_2,\gamma_3$ be positive real numbers and $x_1,x_2,x_3$ integers and let
\[
	0 \neq |\Lambda|
	:= |x_1 \log \gamma_1 + x_2 \log \gamma_2 + x_3 \log \gamma_3|.
\]
Assume that the $x_i$ are bounded in absolute values by some constant $M$ and choose a constant $C>M^3$. 
Consider the matrix
\[
	A= \begin{pmatrix}
	1 & 0 & 0 \\
	0 & 1 & 0 \\
	[C \log \gamma_1] & [C \log \gamma_2] &[C \log \gamma_3]
	\end{pmatrix},
\]
where $[x]$ denotes the nearest integer to $x$. The columns of $A$ form a basis of a lattice. Let $B$ be the matrix that corresponds to the LLL-reduced basis and let $B^*$ be the matrix corresponding to the Gram-Schmidt basis constructed from $B$. Let $c$ be the Euclidean norm of the smallest column vector of $B^*$ and set $S:=2M^2$ and $T:=(1+3M)/2$.
If $c^2 > T^2 + S$, then 
\[
	|\Lambda|
	> \frac{1}{C}\left( \sqrt{c^2-S} - T \right).
\]
\end{lem}

Finally, let us mention some inequalities that will be helpful. We will frequently use the fact that
\[
	|\log (1+x)| < 2 |x|
\]
for all real $x$ with $|x|< 0.5$. 
Also, we have
\[
	\log (a+x) 
	< \log a + \frac{x}{a}
\]
for all $a\geq 1$ and $x>0$.
Moreover, we will use the following lemma.

\begin{lem}\label{lem:ineqLog}
Let $x,c,a$ be positive real numbers with $x\geq 1$, $a\geq 2$ and
\begin{equation}\label{eq:lem:ineqLog1}
	x < c( a + \log x).
\end{equation}
Then
\[
	x < 2c ( a + \log c).
\]
\end{lem}
\begin{proof}
If \eqref{eq:lem:ineqLog1} is satisfied, then taking the logarithm and using $a\geq 2$ we see that
\[
	\log x 
	< \log c + \log (a + \log x)
	< \log c + \log a + \frac{\log x}{2}.
\]
This implies $\log x < 2(\log c + \log a) < 2\log c + a$ and plugging into \eqref{eq:lem:ineqLog1} we obtain the desired result.
\end{proof}

\section{Set up}\label{sec:setup}

In this section we first check solutions for small $n$ as well as $|y|\leq 1$. 

Then
we consider the polynomial $f_n(X)=X(X-F_n)(X-2^n)-1$ and use a result by Thomas \cite{Thomas1979} to find systems of fundamental units in the corresponding orders. We estimate the roots $\alpha^{(i)}$ of $f_n$, the units $\alpha^{(i)}-F_n$ and $\alpha^{(i)}-2^n$ and the regulator. Then for a solution $(x,y)$ we estimate the units $\beta^{(i)}=x-\alpha^{(i)}y$. Finally, for a representation $\beta=\pm  \eta_1^{u_1}\eta_2^{u_2}$ we bound the exponents $u_1,u_2$.

\medskip

For small $n$ we solve Equation \eqref{eq:main} using the functions \verb|thueinit| and \verb|thue| of PARI/GP~\cite{PARI}. For $n\leq 28$ this takes only a couple of minutes on a usual PC and we have

\begin{lem}\label{lem:small-sols}
For $n\leq 28$ the only solutions $(x,y)\in \Z^2$ to Equation \eqref{eq:main} are those given in Theorem \ref{thm:main}.
\end{lem}

From now on, we assume that $n> 28$.
In the course of the paper we will be doing a lot of estimations. 
In order to keep track of constants, we use the following $L$-notation: For real functions $f(n), g(n)$ with $g>0$ we write $f(n)=L(g(n))$ if $|f(n)|\leq g(n)$ for all $n \geq 29$.
Note that after finishing Section \ref{sec:uBy-n1000} we will assume $n>1000$. Then statements involving the $L$-notation will only need to hold for $n>1000$.

\begin{lem}\label{lem:y-small}
The solutions $(x,y)\in \Z^2$ to Equation \eqref{eq:main} with $|y|\leq 1$ are precisely the trivial solutions given in \eqref{eq:trivial_sols}.
\end{lem}
\begin{proof}
If $y=0$, then \eqref{eq:main} becomes $x^3=\pm 1$, which yields the solution $ (\pm 1,0)$. If $|y|=1$, then either $x(x\pm F_n)(x\pm 2^n)=0$ or $x(x\pm F_n)(x\pm 2^n)=\pm 2$. The first case yields the solutions  $(0,\pm 1), \mp(F_n,1), \mp(2^n,1)$. In the second case we see that for $n\geq 3$ the difference between the factors $x$ and $x\pm 2^n$ is at least $8$ so the product cannot be a divisor of 2.
\end{proof}

From now on, we are only interested in solutions $(x,y)$ with $|y|\geq 2$.

Consider the polynomial corresponding to \eqref{eq:main} at $Y=1$:
\[
	f_n(X):= X (X- F_n) (X- 2^n) - 1.
\]
Polynomials of this form, and in particular unit structures of related orders, have been studied e.g. in \cite{BernsteinHasse1969}, \cite{Stender1972} and \cite{Thomas1979}. The following lemma is due to Thomas~\cite[Theorem 3.9]{Thomas1979}.

\begin{lem}[Thomas]\label{lem:thomas}
Let $f(X)=X(X-r)(X-s)-1$ with integers $2\leq r \leq s-2$. Then $f(X)$ is irreducible with three distinct real roots. If $\alpha$ is a zero of $f(X)$, then $\{\alpha, \alpha - r \}$ is a fundamental system of units of the order $\Z[\alpha]$.
\end{lem}
Note that since in the lemma we have $\alpha (\alpha - r) (\alpha - s) =1$, any two of the three units $\alpha, \alpha-r, \alpha - s$ form a fundamental system of units of $\Z[\alpha]$.

Let us denote the roots of $f_n$ by $\alpha^{(i)}$ $(i=1,2,3)$ with $\alpha^{(1)}<\alpha^{(2)}<\alpha^{(3)}$.
Moreover, let
$G_1:=0$, $G_2:=F_n$ and $G_3=2^n$, so that
\[
	f_n(X)= (X-G_1) (X- G_2) (X- G_3) - 1.
\] 
Now set 
\[
	\eta_i^{(k)}:=\alpha^{(k)} - G_i \qquad (i,k=1,2,3).
\] 
Then by Lemma \ref{lem:thomas}, any two of $\eta_1^{(k)}, \eta_2^{(k)},\eta_3^{(k)}$ form a system of fundamental units of $\Z[\alpha^{(k)}]$.
Finally, let us denote the logarithms by
\[
	l_i^{(k)}:=\log |\eta_i^{(k)}|.
\]
Note that all these terms implicitly depend on $n$.
We estimate the units $\eta_i^{(k)}$ and in particular the roots of $f_n$:

\begin{lem}\label{lem:roots-estimate}
The roots of $f_n$ can be estimated by
\[
	\alpha^{(i)}= G_i + L(0.5^n)
	\qquad (i=1,2,3).
\]
\end{lem}
\begin{proof}
Note that $f_n(0)=-1$ and 
\[
	f_n(0.5^n)
	= 0.5^n(0.5^n - F_n)(0.5^n - 2^n)-1 
	>0.
\]
Therefore, a zero must lie between $0$ and $0.5^n$, i.e. one of the zeros can be estimated by $G_1+L(0.5^n)$.

Similarly, $f_n(F_n)=f_n(2^n)=-1$ and one can check that $f_n(F_n-0.5^n)>0$ and $f_n(2^n + 0.5^n)>0$. Thus, the other two roots are given by $G_2 + L(0.5^n)$ and $G_3+L(0.5^n)$.
\end{proof}

\begin{lem}\label{lem:lik}
We have
\[
	l_i^{(k)}=
	\begin{cases}
	n \log 2 + L(0.81^n) & \text{if } i\neq k \text{ and } \max\{i,k\}=3,\\
	n \log \phi - \log \sqrt{5} + L(2\cdot 0.81^n) & \text{if } i\neq k \text{ and } \max\{i,k\}=2,\\
	-n (\log \phi + \log 2) + \log \sqrt{5} + L(3\cdot 0.81^n) & 	\text{if } i=k \in\{1,2\},\\
	-2n\log 2 + L(2\cdot 0.81^n) & \text{if } i=k=3.
	\end{cases}
\]
\end{lem}
\begin{proof}
For $i\neq k$ we have
\begin{align}
	l_i^{(k)}
	&= \log |\alpha^{(k)}-G_i|
	= \log |G_k + L(0.5^n) - G_i|\nonumber\\
	&= \log \left| G_{\max\{i,k\}}
		\left(1- \frac{G_{\min\{i,k\}}}{G_{\max\{i,k\}}} + \frac{L(0.5^n)}{G_{\max\{i,k\}}} \right)\right|\nonumber\\
	&= \log |G_{\max\{i,k\}}| + L(0.81^n),\label{eq:lik}
\end{align}
where we used that
\begin{align*}
	\frac{G_{\min\{i,j\}}}{G_{\max\{i,j\}}} + \frac{L(0.5^n)}{G_{\max\{i,j\}}}
	&\leq \frac{G_2}{G_{3}} + \frac{0.5^n}{G_3}\\
	&= \frac{\phi^n - \psi^n + \sqrt{5}\cdot 0.5^n }{\sqrt{5} \cdot 2^n}
	< \frac{1}{2} \cdot 0.81^n,
\end{align*}
where the constant $0.81$ comes from $\phi/2\approx 0.809<0.81$.

If $i\neq k \text{ and } \max\{i,k\}=3$, then by \eqref{eq:lik} 
\[
l_i^{(k)}= \log 2^n + L(0.81^n)= n \log 2 + L(0.81^n).
\]
If $i\neq k \text{ and } \max\{i,j\}=2$, then
\begin{align*}
	l_i^{(k)}
	&=\log F_n + L(0.81^n)
	=\log \frac{\phi^n-\psi^n}{\sqrt{5}} + L(0.81^n)\\
	&= \log \phi^n + \log (1- (\psi/\phi)^n) - \log \sqrt{5} + L(0.81^n)\\	
	&= n \log \phi + L(2 \cdot 0.39^n) - \log \sqrt{5} + L(0.81^n)\\
	&= n \log \phi - \log \sqrt{5} + L(2\cdot 0.81^n).
\end{align*}
The estimates for $i=k$ then follow immediately from the fact that $\eta_1^{(k)}\eta_2^{(k)}\eta_3^{(k)}= (\alpha^{(k)}-G_1)(\alpha^{(k)}-G_2)(\alpha^{(k)}-G_3) = 1$ and therefore $l_1^{(k)} + l_2^{(k)} + l_3^{(k)} = 0$.
\end{proof}

Now we can easily estimate the regulator of the unit group of $\Z[\alpha^{(i)}]$:

\begin{lem}\label{lem:reg_index}
Let $i,j,k,l \in \{1,2,3\}$ with $i\neq j, k\neq l$ and let
\begin{align*}
	R
	&:= \left| \det
	\begin{pmatrix}
		\log|\eta_i^{(k)}| &|\log \eta_j^{(k)}| \\
		\log |\eta_{i}^{(l)}| &  |\log \eta_{j}^{(l)}|
	\end{pmatrix} \right|
	= \left| \det \begin{pmatrix}
		l_i^{(k)} &  l_j^{(k)} \\
		l_{i}^{(l)} & l_{j}^{(l)} 
	\end{pmatrix}\right|
	= \left| l_i^{(k)}l_{j}^{(l)}-l_{i}^{(l)}l_j^{(k)} \right|.
\end{align*}
Then the value of $R$ is independent of the choice of $i,j,k,l$ and we have
\[
	n^2 < R < 2 n^2.
\]
\end{lem}
\begin{proof}
The independence of the choice of $i,j,k,l$ follows from the unit group structure: Since
$\eta_1^{(k)}\eta_2^{(k)}\eta_3^{(k)}=1$ and $\eta_i^{(1)}\eta_i^{(2)}\eta_i^{(3)}= \Norm_{\Q(\alpha^{(1)})/\Q}(\eta_i^{(1)})=\pm 1$, we have
\[
	\begin{pmatrix}
		l_1^{(k)}\\
		l_1^{(l)} 
	\end{pmatrix}
	+
	\begin{pmatrix}
		l_2^{(k)}\\
		l_2^{(l)} 
	\end{pmatrix}
	+
	\begin{pmatrix}
		l_3^{(k)}\\
		l_3^{(l)} 
	\end{pmatrix}
	= 0 
	\quad \text{and} \quad
	\begin{pmatrix}
		l_i^{(1)} & l_j^{(1)} 
	\end{pmatrix}
	+
	\begin{pmatrix}
		l_i^{(2)} & l_j^{(2)} 
	\end{pmatrix}
	+
	\begin{pmatrix}
		l_i^{(3)} & l_j^{(3)} 
	\end{pmatrix}
	=0.
\]
Thus we may exchange columns and lines in the matrix in the lemma, changing at most the sign of the determinant.

Using Lemma \ref{lem:lik} we compute
\begin{align*}
	\pm R
	&= \det \begin{pmatrix}
		l_1^{(1)} & l_2^{(1)} \\
		l_1^{(2)} & l_2^{(2)}
	\end{pmatrix}
	= l_1^{(1)}l_{2}^{(2)}-l_{1}^{(2)}l_2^{(1)}\\
	&= (- n(\log \phi + \log 2) + \log \sqrt{5} + L(3\cdot 0.81^n))^2 
		- (n\log \phi - \log \sqrt{5} + L(2\cdot 0.81^n))^2\\
	&= n^2 ( 2\log \phi \log 2 + (\log 2)^2 ) 
		- n \cdot 2 \log 2 \log \sqrt{5}
		+ L(10n \cdot 0.81^n)\\
	&= n^2 ( 2\log \phi \log 2 + (\log 2)^2 ) + L(1.2 n)
	\approx 1.15 n^2 + L(1.2n)
\end{align*}
and the inequalities follow.
\end{proof}

Now let $(x,y)\in \Z^2$ be a solution of \eqref{eq:main} with $|y|\geq 2$ and let $\alpha=\alpha^{(1)}$. Then we set 
\[
	\beta:= x-\alpha y.
\] 
This is a unit in $\Z[\alpha]$, since
\[
	\Norm_{\Q(\alpha)/\Q}(x-\alpha y)
	= (x- \alpha^{(1)} y) (x- \alpha^{(2)} y) (x-\alpha^{(3)}y)
	= x(x-F_n y)(x-2^n y) - y^3
	= \pm 1.
\]
We denote the conjugates of $\beta$ as usual by $\beta^{(i)}=x - \alpha^{(i)} y$ $(i=1,2,3)$ and we define the \textit{type}~$j$ of a solution $(x,y)\in \Z$ of \eqref{eq:main} by
\[
	|\beta^{(j)}|=\min_{i=1\bb n} |\beta^{(i)}|.
\]
From now on let $j$ be the type of $(x,y)$ and $\{j,k,l\}=\{1,2,3\}$.
We collect some information on the size of the $\beta^{(i)}$:

\begin{lem}\label{lem:betaiuj}
We have
\begin{align}
	|\beta^{(j)}|
	&\leq \frac{4}{|y|^2 \cdot|\alpha^{(l)} - \alpha^{(j)}| \cdot |\alpha^{(k)} - \alpha^{(j)}|},\label{eq:lem:betaiuj-j}\\
	|\beta^{(i)}|
	&\geq \frac{|y| \cdot |\alpha^{(i)}-\alpha^{(j)}|}{2}
	\quad \text{for } i =k,l\label{eq:lem:betaiuj-i}
\end{align}
and
\begin{equation}\label{eq:lem:betaiuj-log}
	\log |\beta^{(i)}| = \log |y| + l_j^{(i)} + L(0.4^n)
	\quad \text{for } i =k,l.
\end{equation}
\end{lem}
\begin{proof}
For $i\neq j$ we have 
\[
	|y| \cdot |\alpha^{(i)}- \alpha^{(j)}| 
	=|\alpha^{(i)} y - x + x -  \alpha^{(j)} y|
	\leq |\beta^{(i)}|+|\beta^{(j)}|
	\leq 2|\beta^{(i)}|,
\]
which proves Inequality \eqref{eq:lem:betaiuj-i}.
Since $\beta^{(1)}\beta^{(2)}\beta^{(3)}=\pm 1$, we get
\begin{align}\label{eq:betaj}
	|\beta^{(j)}|
	= \frac{1}{|\beta^{(k)}||\beta^{(l)}|}
	\leq \frac{4}{|y|^{2}|\alpha^{(l)}- \alpha^{(j)}||\alpha^{(k)}- \alpha^{(j)}| }
	\leq \frac{1}{|\alpha^{(l)}- \alpha^{(j)}||\alpha^{(k)}- \alpha^{(j)}|}
\end{align}
and we have proven Inequality \eqref{eq:lem:betaiuj-j}.
Moreover, using Lemma \ref{lem:roots-estimate}, it is easy to see that 
\begin{align*}
	|\alpha^{(l)}- \alpha^{(j)}||\alpha^{(k)}- \alpha^{(j)}|
	&\geq |\alpha^{(1)}- \alpha^{(2)}||\alpha^{(3)}- \alpha^{(2)}|\\
	&\geq (F_n - 2\cdot 0.5^n)) (2^n - F_n - 2 \cdot 0.5^n)\\
	&> 2^n.
\end{align*}
Thus we get from \eqref{eq:betaj}
\begin{equation}\label{eq:betaj2}
\beta^{(j)} = L(0.5^n).
\end{equation}
For $i=k,l$ we have
\begin{equation*}
	\left| \frac{\beta^{(i)}}{y}\right|
	= \left| 
		\frac{x}{y} - \alpha^{(j)} + \alpha^{(j)} - G_j + G_j - \alpha^{(i)}
	  \right|
\end{equation*}
and since $x/y-\alpha^{(j)}= \beta^{(j)}/y=L(0.5^n)$ by \eqref{eq:betaj2}, 
$\alpha^{(j)} - G_j = L(0.5^n)$ by Lemma~\ref{lem:roots-estimate} and $G_j - \alpha^{(i)}= \eta_j^{(i)}$, we get
\[
	\left| \frac{\beta^{(i)}}{y}\right| 
	= |\eta_j^{(i)}| + L(2\cdot 0.5^n).
\]
Since $|\eta_j^{(i)}| \geq F_n - 0.5^n > 1.5^n$, this yields
\[
	\left| \frac{\beta^{(i)}}{y}\right| 
		= |\eta_j^{(i)}|\left( 1 + L\left(\frac{2\cdot 0.5^n}{|\eta_j^{(i)}|}\right) \right)
	= |\eta_j^{(i)}| \left( 1+ L\left(\frac{1}{2} \cdot 0.4^n\right) \right).
\]
Taking logarithms we obtain Inequality \eqref{eq:lem:betaiuj-log}.
\end{proof}

Finally, note that since $\eta_1=\eta_1^{(1)}$ and $\eta_2=\eta_2^{(1)}$ are fundamental units in $\Z[\alpha]$, we can write $\beta=\beta^{(1)}$ as
\begin{equation}\label{eq:darstellungBeta}
	\beta = \pm \eta_1^{u_1} \eta_2^{u_2}.
\end{equation}
We find an upper bound for $u_1$ and $u_2$:

\begin{lem}\label{lem:bound_u}
For $u_1,u_2$ from \eqref{eq:darstellungBeta} we have
\[
	\max\{|u_1|,|u_2|\}
	< \frac{2\log|y|}{n} + 2.	 
\]
\end{lem}
\begin{proof}
We take the absolute value and the logarithm of \eqref{eq:darstellungBeta} and consider the two conjugates which do not correspond to the type $j$ of $(x,y)$. We get
\begin{align*}
	\log |\beta^{(k)}| &= u_1 \log |\eta_1^{(k)}| + u_2 \log |\eta_2^{(k)}|,\\
	\log |\beta^{(l)}| &= u_1 \log |\eta_1^{(l)}| + u_2 \log |\eta_2^{(l)}|,
\end{align*}
which we can rewrite as 
\[
	\begin{pmatrix}
		\log |\beta^{(k)}|\\
		\log |\beta^{(l)}|
	\end{pmatrix}
	=
	\begin{pmatrix}
		l_1^{(k)} & l_2^{(k)}\\
		l_1^{(l)} & l_2^{(l)}
	\end{pmatrix}
	\begin{pmatrix}
	u_1\\
	u_2
	\end{pmatrix}.
\]
We denote the above $(2 \times 2)$-matrix by $A$. Then $A$ has exactly the form of the matrix in Lemma~\ref{lem:reg_index}. In particular, it has non-zero determinant and we can apply its inverse obtaining
\[
	\begin{pmatrix}
	u_1\\
	u_2
	\end{pmatrix}
	=	
	A^{-1} 
	\cdot 
	\begin{pmatrix}
		\log |\beta^{(k)}|\\
		\log |\beta^{(l)}|
	\end{pmatrix}.
\]
Taking the maximum absolute row sum norm $\|\cdot\|_\infty$ on both sides, we see that
\begin{equation}\label{eq:maxu1u2}
	\max\{|u_1|,|u_2|\}
	\leq \| A^{-1}\|_\infty \max\{ \log |\beta^{(k)}|, \log |\beta^{(l)}| \}.
\end{equation}
The inverse of $A$ is given by
\[
	A^{-1}
	= \frac{1}{\det A} 
	\begin{pmatrix}
		l_2^{(l)} & -l_2^{(k)}\\
		-l_1^{(l)} & l_1^{(k)}
	\end{pmatrix}.
\]
Since $|\det A|=R > n^2$ by Lemma \ref{lem:reg_index}, we can use Lemma \ref{lem:lik} to estimate the absolute row sum norm of $A^{-1}$:
\begin{align}\label{eq:A-1}
	\|A^{-1}\|_\infty
	&< \frac{1}{n^2} (n(\log \phi + \log 2)- \log \sqrt{5} + L(3\cdot 0.81^n) + n \log 2 + L(0.81))
	< 2 n^{-1}.
\end{align}
Combining \eqref{eq:maxu1u2}, \eqref{eq:A-1}, \eqref{eq:lem:betaiuj-log} and Lemma \ref{lem:lik}
we obtain
\begin{align*}
	\max\{|u_1|,|u_2|\}
	&< 2n^{-1} ( \log |y| + \max\{|l_j^{(k)}|,|l_j^{(l)}|\} + L(0.4^n) )\\
	&< 2n^{-1} ( \log |y| + 0.7 n)
	= \frac{2\log|y|}{n} + 2.	 
\end{align*}
\end{proof}

\section{An upper bound for $\log |y|$ and solutions for $n\leq 1000$}\label{sec:uBy-n1000}

In this section we do the standard procedure for solving Thue equations and bound $|y|$ in terms of $n$. Note that such a bound can be obtained immediately from an explicit result by Bugeaud and Gy\H{o}ry \cite{BugeaudGyory1996}. However, we want slightly better constants and we also want to completely solve Equation \eqref{eq:main} for $n\leq 1000$.

\begin{lem}\label{lem:upperBoundY}
For $n>28$ we have
\[
	\log |y| 
	< 3.66 \cdot 10^{16} \cdot n^3.
\]
\end{lem}
\begin{proof}
Note that the identity
\[
	(\alpha^{(j)}-\alpha^{(k)})\beta^{(l)} +
	(\alpha^{(l)}-\alpha^{(j)})\beta^{(k)} +
	(\alpha^{(k)}-\alpha^{(l)})\beta^{(j)} 
	=0
\]
implies
\begin{equation}\label{eq:nachSiegel}
	\frac{\alpha^{(j)}-\alpha^{(k)}}{\alpha^{(j)}-\alpha^{(l)}} \cdot \frac{\beta^{(l)}}{\beta^{(k)}} -1 
	= \frac{\alpha^{(l)}-\alpha^{(k)}}{\alpha^{(j)}-\alpha^{(l)}} \cdot \frac{\beta^{(j)}}{\beta^{(k)}}.
\end{equation}
We estimate the right hand side of \eqref{eq:nachSiegel} using \eqref{eq:lem:betaiuj-j} and \eqref{eq:lem:betaiuj-i} from Lemma \ref{lem:betaiuj}. We also use Lemma \ref{lem:roots-estimate} to estimate the differences between the roots $\alpha^{(i)}$:
\begin{align}
	\left| \frac{\alpha^{(l)}-\alpha^{(k)}}{\alpha^{(j)}-\alpha^{(l)}} \cdot \frac{\beta^{(j)}}{\beta^{(k)}} \right|
&\leq  
	\frac{|\alpha^{(l)}-\alpha^{(k)}|}{|\alpha^{(j)}-\alpha^{(l)}|}
	\cdot 
	\frac{4}{|y|^{2}|\alpha^{(l)}- \alpha^{(j)}||\alpha^{(k)}- \alpha^{(j)}| }
	\cdot \frac{2}{|y| \cdot |\alpha^{(k)}- \alpha^{(j)}|}\nonumber\\
&=
	\frac{8}{|y|^3} \cdot \frac{|\alpha^{(l)}-\alpha^{(k)}|}{|\alpha^{(j)}- \alpha^{(l)}|^2|\alpha^{(j)}- \alpha^{(k)}|^2}\nonumber\\
&\leq
	\frac{8}{|y|^3} \cdot \frac{|\alpha^{(3)}-\alpha^{(1)}|}{|\alpha^{(2)}- \alpha^{(3)}|^2|\alpha^{(2)}- \alpha^{(1)}|^2}\nonumber\\
&\leq \frac{8}{|y|^3}\cdot \frac{2^n + 2\cdot 0.5^n}{(2^n-F_n-2\cdot 0.5^n)^2(F_n - 2\cdot 0.5^n)^2}
	< 40 \cdot 0.2^n \cdot |y|^{-3}.\label{eq:nachSiegel-RS}
\end{align}
Next we use the representation for $\beta$ from \eqref{eq:darstellungBeta} and rewrite the left hand side of \eqref{eq:nachSiegel} as
\begin{equation}\label{eq:nachSiegel-LS}
	\frac{\alpha^{(j)}-\alpha^{(k)}}{\alpha^{(j)}-\alpha^{(l)}} \cdot \frac{\beta^{(l)}}{\beta^{(k)}} -1 
= \frac{\alpha^{(j)}-\alpha^{(k)}}{\alpha^{(j)}-\alpha^{(l)}} \cdot 
	\left(\frac{\eta_1^{(l)}}{\eta_1^{(k)}}\right)^{u_1}
	\left(\frac{\eta_2^{(l)}}{\eta_2^{(k)}}\right)^{u_2}
	-1.
\end{equation}
Since $|\log x| < 2 |x-1|$ for $|x-1|<0.5$, we obtain from \eqref{eq:nachSiegel}, \eqref{eq:nachSiegel-RS} and \eqref{eq:nachSiegel-LS}
\begin{equation}\label{eq:lambda-upperBound}
	|\Lambda|:=	
	\left| 
	\log \left|\frac{\alpha^{(j)}-\alpha^{(k)}}{\alpha^{(j)}-\alpha^{(l)}}\right|
	+ u_1 \log \left| \frac{\eta_1^{(l)}}{\eta_1^{(k)}}\right|
	+ u_2 \log \left|\frac{\eta_2^{(l)}}{\eta_2^{(k)}}\right|
	 \right|
<
	80 \cdot 0.2^n \cdot |y|^{-3}.
\end{equation}
Note that $\Lambda$ is zero if and only if \eqref{eq:nachSiegel-LS} is zero. But \eqref{eq:nachSiegel-LS} is only zero if \eqref{eq:nachSiegel} is zero, which is not the case because the right hand side is clearly non-zero. Thus $\Lambda \neq 0$.

In order to apply a lower bound for linear forms in logarithms, we need to estimate the heights of the numbers in the logarithms, let us call these numbers $\gamma_0,\gamma_1,\gamma_2$.
The height of the roots $\alpha^{(i)}$ is given by
\[
	h(\alpha^{(i)}) 
	= \frac{1}{3}(\log |\alpha^{(2)}| + \log |\alpha^{(3)}|)
	< 0.4 \cdot n
	\quad (i=1,2,3),
\]
where we used Lemma \ref{lem:roots-estimate} for the estimation.
By the properties of heights we have
\begin{align*}
	h(\gamma_0)=
	h \left( \left| \frac{\alpha^{(j)}-\alpha^{(k)}}{\alpha^{(j)}-\alpha^{(l)}} \right| \right) 
	&\leq 2 h(\alpha^{(j)}) + h(\alpha^{(k)}) + h(\alpha^{(l)}) + 2\log 2
	< 1.7 n
\end{align*}
and
\begin{align*}
	h(\gamma_i)=
	h \left( \frac{\eta_i^{(l)}}{\eta_i^{(k)}}\right)
	&= h \left( \left| \frac{\alpha^{(l)}-G_i}{\alpha^{(k)}-G_i}\right|\right)
	\leq 2 h(G_i) + h(\alpha^{(l)}) + h(\alpha^{(k)}) + 2 \log 2\\
	&< 2 \cdot n \log 2 + 2 \cdot 0.4 \cdot n + 2 \log 2
	< 2.3 n \quad (i=1,2).
\end{align*}
By Lemma \ref{lem:bound_u} we may set
\[
	B:= \frac{2 \log|y|}{n}+2.
\]
Note that $D=[\Q(\gamma_0,\gamma_1,\gamma_2):\Q] \leq 6$, so an
application of Lemma~\ref{lem:matveev} yields
\[
	\log |\Lambda| 
	> -1.4 \cdot 10^{11} \cdot 6^2 \cdot \log (e \cdot 6) \cdot \log \left(e \cdot \left(\frac{2 \log|y|}{n} + 2 \right)\right)\cdot 6^3 \cdot 1.7n \cdot 2.3n \cdot 2.3n.
\]
If $\log |y|/n <4$, then we are already done. Thus we may assume $\log |y|/n \geq 4$ and estimate $\log (e ( 2\log |y| /n + 2))\leq \log \log |y| - \log n + 2$ and obtain
\[
	\log |\Lambda|
	> -C_1 (\log \log |y| - n + 2)n^3,
\]
with $C_1 = 2.74\cdot 10^{16}> 1.4 \cdot 10^{11} \cdot 6^2 \cdot \log (e \cdot 6)\cdot 6^3 \cdot 1.7 \cdot 2.3 \cdot 2.3$.
Thus, combined with \eqref{eq:lambda-upperBound} we have
\[
	- C_1 (\log \log |y| - \log n + 2)n^3
	< \log 80 + n \log 0.2 - 3 \log |y|.
\]
This implies 
\[
	\log |y|
	< \frac{C_1}{3} (2n^3 + \log \log |y|)
\]
and with Lemma \ref{lem:ineqLog} we get 
\[
	\log |y| 
	< \frac{2C_1}{3} \left(2 n^3 + \log \frac{C_1}{3}\right)
	< 	3.66 \cdot 10^{16} \cdot n^3.
\]
\end{proof}

\begin{lem}\label{lem:solutions10n1000}
For $28 < n \leq 1000$ the only solutions to \eqref{eq:main} are the trivial solutions.
\end{lem}
\begin{proof}
Using Sage \cite{sagemath}, we do the following procedure for each $n= 29\bb 1000$.
First, we use Lemma \ref{lem:upperBoundY} to compute an upper bound for $\log |y|$:
\[
	L:=3.66 \cdot 10^{16} \cdot n^3. 
\]
Then we compute the bound for $|u_1|,|u_2|$ from Lemma \ref{lem:bound_u}:
\[
	M:=\frac{2 L}{n}+2.
\]
Next, we recall Inequality \eqref{eq:lambda-upperBound}:
\begin{equation*}
	\left| 
	\log \gamma_0
	+ u_1 \log \gamma_1
	+ u_2 \log \gamma_2
	 \right|
<
	80 \cdot 0.2^n \cdot |y|^{-3},
\end{equation*}
where
\begin{equation*}
	\gamma_0 = \left|\frac{\alpha^{(j)}-\alpha^{(k)}}{\alpha^{(j)}-\alpha^{(l)}}\right|,
	\quad
	\gamma_1 = \left| \frac{\eta_1^{(l)}}{\eta_1^{(k)}}\right|,
	\quad
	\gamma_2 = \left|\frac{\eta_2^{(l)}}{\eta_2^{(k)}}\right|,
\end{equation*}
and use the Baker-Davenport reduction method to reduce the bound on $|y|$.
In the following, we distinguish between the three cases $j=1,2,3$ (note that we can then choose the values for $k$ and $l$ freely, and switching the values of $k,l$ does not change the reduction process anyway). 
In each case the numbers $\gamma_0, \gamma_1, \gamma_2$ depend on $n$ and can be computed to arbitrarily high precision (since the the roots $\alpha^{(1)},\alpha^{(2)},\alpha^{(3)}$ of $f_n(X)$ can be determined numerically).
After finding a suitable convergent $\frac{p}{q}$ to $\frac{\log \gamma_1}{\log \gamma_2}$, such that the $\eps$ in Lemma~\ref{lem:Baker-Davenport} is positive, we compute the upper bound for $|y|$:
\[
	Y:= \left( \frac{q \cdot 80 \cdot 0.2^n}{\eps |\log \gamma_2|} \right)^{1/3}.
\]
Indeed, we mange to find such a convergent (and therefore a bound $Y$) in each case and for each $n$. 
Then we note that by Lemma \ref{lem:betaiuj} we have
\[
	|x- \alpha^{(j)}y|
	= |\beta^{(j)}|
	\leq \frac{4}{|y|^2 |\alpha^{(l)}- \alpha^{(j)}| \cdot |\alpha^{(k)}-\alpha^{(j)}|}
	< \frac{1}{2 |y|},
\]
which implies that $\frac{x}{y}$ is a convergent to $\alpha^{(j)}$ (see e.g. \cite[p. 47]{Baker1984}). Thus we only need to check all the convergents $\frac{x}{y}$ to $\alpha^{(j)}$ with $2 \leq y < Y$.
The size of $Y$ varies with the values of $j,k,l$ and $n$ and lies between $4$ and $3.29\cdot 10^{99}$ in each case. 
We ran all the computations in Sage for $n=29\bb 1000$ and all three cases $j=1,2,3$. The computations took about one hour on a usual PC and revealed no solutions.
\end{proof}

\begin{rem}
In the above computations the size of $Y$ varies considerably. Not only is there a general increase in $Y$ as $n$ grows, but also does the case $j=3$ yield much larger values than the cases $j=1,2$. 

Usually, in the Baker-Davenport reduction  it is enough to find a convergent $p/q$ with $q>6M$ for $\eps$ to be positive, and $\eps$ is a rather random small positive number. Thus, at first glance, we would expect $Y$ to get smaller and smaller, because $q$ and $\log \gamma_2$ are roughly linear in $n$ and the factor $0.2^n$ should make $Y$ very small. 

However, if $j=1$ or $j=2$, then it turns out that $\gamma_0 \approx \frac{1}{\gamma_1}$, so $\frac{\log \gamma_1}{\log \gamma_2}$ and $\frac{\log \gamma_0}{\log \gamma_2}$ have very similar size (and opposite sign). Thus, $q$ must be very large if
\[
	\eps =\left\| \frac{\log \gamma_0}{\log \gamma_2} q \right\| 
		- M \left\| \frac{\log \gamma_1}{\log \gamma_2} q \right\|
\] 
should be positive. And even when $\eps$ is positive, it will be very small (increasing $q$ further makes $\eps$ larger, but does not help reduce $Y$). This is the reason that $Y$ is surprisingly large and the effect gets more significant as $n$ grows. For $n=29$ we get approximately $Y= 4.18$ for $j=1$ and $Y= 4.74$ for $j=2$. For $n=1000$ we get $Y=1.47\cdot 10^7$ and $Y=1.30\cdot 10^7$. 

If $j=3$, then the situation is even worse: $\gamma_0$ gets very close to 1 and $\frac{\log \gamma_0}{\log \gamma_2}$ gets very close to zero as $n$ increases. Moreover, $\gamma_1\approx \frac{1}{\gamma_2}$, so $\frac{\log \gamma_1}{\log \gamma_2}$ gets very close to $-1$. It turns out that we need extremely large $q$ for $\eps$ to be positive. For $n=29$ we get approximately $Y= 5670$, and for $n=1000$ we get $Y= 3.28 \cdot 10^{99}$. 

In the next section we will see again that there seems to be a fundamental difference between the cases $j=1,2$ and the case $j=3$. But there, $j=3$ will be the easier scenario.
\end{rem}

From now on assume that $n>1000$.

\section{A lower bound for $\log |y|$}\label{sec:lowerBoundLogY}

In this section we find a lower bound for $\log |y|$ depending on $n$ and on the type of the solution $(x,y)$. 

\begin{lem}\label{lem:j3}
If $n>1000$ and the type of $(x,y)$ is $j=3$, then 
\[
	\log |y| > 1.2^n.
\]
\end{lem}
\begin{proof}
Assume that $j=3$.
Since $\beta$ is a unit in $\Z[\alpha]$ and $\eta_2=\eta_2^{(1)},\eta_3=\eta_3^{(1)}$ are a fundamental system of units of $\Z[\alpha]$, we can write
\[
	\beta = \pm \eta_2^{u_2} \eta_3^{u_3}.
\] 
Taking absolute values and the logarithm and considering the conjugates, we obtain
\begin{align*}
	\log |\beta^{(1)}| = u_2 l_2^{(1)} + u_3 l_3^{(1)},\\
	\log |\beta^{(2)}| = u_2 l_2^{(2)} + u_3 l_3^{(2)}.
\end{align*}
We solve the system of linear equations for $u_2$ and get
\[
	\pm R u_2	
	=u_2(l_2^{(1)} l_3^{(2)} - l_2^{(2)} l_3^{(1)})
	=  l_3^{(2)} \log|\beta^{(1)}| - l_3^{(1)}\log|\beta^{(2)}|.
\]
Plugging in Equation \eqref{eq:lem:betaiuj-log} from Lemma \ref{lem:betaiuj} yields
\begin{align}
	\pm R u_2
	&= l_3^{(2)} (\log |y| + l_3^{(1)} + L(0.4^n))
		- l_3^{(1)} (\log |y| + l_3^{(2)} + L(0.4^n))\nonumber\\
	&= \log |y| ( l_3^{(2)} - l_3^{(1)}) + L( 2\cdot 0.7n \cdot 0.4^n)
	= \log |y| ( l_3^{(2)} - l_3^{(1)}) + L( 0.5^n), \label{eq:u2R}
\end{align}
where we used Lemma \ref{lem:lik} to estimate $|l_3^{(2)}|,|l_3^{(1)}|\leq 0.7 n$ in the $L$-term.

In order to show that Equation \eqref{eq:u2R} is not zero, we estimate $|l_3^{(2)} - l_3^{(1)}|$ from below using Lemma \ref{lem:roots-estimate}:
\begin{align*}
	l_3^{(1)} - l_3^{(2)}
	&= \log| \alpha^{(1)}-G_3| - \log | \alpha^{(2)} - G_3|\\
	&= \log | L(0.5^n) - 2^n| - \log| F_n + L(0.5^n) - 2^n|\\
	&= \log (2^n ( 1 + L(0.25^n)) - \log \left( 2^n \left( 1 - \frac{F_n}{2^n} + L(0.25^n) \right)\right)\\
	&= \log ( 1 + L(0.25^n)) - \log \left( 1 - \frac{F_n}{2^n} + L(0.25^{n})\right)\\
	&= \log ( 1 + L(0.25^n)) - \log \left ( 1- \frac{1}{\sqrt{5}}\left(\frac{\phi}{2}\right)^n + L(0.31^n) \right)\\
	&> - 2\cdot 0.25^{n} - \frac{1}{2} \left(  - \frac{1}{\sqrt{5}}\left(\frac{\phi}{2}\right)^n + 0.31^n \right)
	> 0.80^n.
\end{align*}
Note that the constant $0.31$ comes from $|\psi|/2\approx 0.309<0.31$ and
$0.80$ comes from $\phi/2 \approx 0.809 > 0.80$.
Now since $|y|\geq 2$, the right hand side of \eqref{eq:u2R} must be non-zero. Thus the integer $u_2$ is non-zero and we get from \eqref{eq:u2R}
\begin{equation}\label{eq:Rcasej3}
	R
	\leq |\pm R u_2|
	=  \log|y| | l_3^{(2)} - l_3^{(1)}| + L(0.5^n).
\end{equation}
Now we estimate $| l_3^{(2)} - l_3^{(1)}|$ from above. As in the above computations we have
\begin{align*}
	| l_3^{(2)} - l_3^{(1)}|	
	=l_3^{(1)} - l_3^{(2)}
	&= \log ( 1 + L(0.25^n)) - \log \left ( 1- \frac{1}{\sqrt{5}}\left(\frac{\phi}{2}\right)^n + L(0.31^n) \right)\\
	&< 0.25^n - 2 \left(-\frac{1}{\sqrt{5}}\left(\frac{\phi}{2}\right)^n - 0.31^n \right)
	< 0.81^n.
\end{align*}	
With this estimate and Lemma \ref{lem:reg_index}, Inequality \eqref{eq:Rcasej3} implies
\[
	\log |y|
	\geq \frac{R-0.5^n}{| l_3^{(2)} - l_3^{(1)}|}
	> \frac{n^2 - 0.5^n}{0.81^n}
	> 1.2^n.
\]
\end{proof}

\begin{lem}\label{lem:j12}
If $n>1000$ and the type of $(x,y)$ is $j\in\{1,2\}$, then
\[
	\log |y|
	> \exp\left( \frac{n}{5.7 \cdot 10^{12}} - \log n - 4 \right).
\]
\end{lem}
\begin{proof}
This time we consider the fundamental units $\eta_1=\eta_1^{(1)}$ and $\eta_2=\eta_2^{(1)}$ and write
\[
	\beta = \pm \eta_1^{u_1} \eta_2^{u_2},
\]
as in Equation \eqref{eq:darstellungBeta}.
We distinguish between the two cases $j=1$ and $j=2$.

\noindent\textbf{Case 1:} $j=1$.
We take absolute values, logarithms and consider the second and third conjugate:
\begin{align*}
	\log |\beta^{(2)}| = u_1 l_1^{(2)} + u_2 l_2^{(2)},\\
	\log |\beta^{(3)}| = u_1 l_1^{(3)} + u_2 l_2^{(3)}.
\end{align*}
Solving the system of linear equations for $u_1$ we obtain
\begin{equation}\label{eq:u1}
	(l_1^{(2)} l_2^{(3)} - l_1^{(3)} l_2^{(2)}) u_1	
	= l_2^{(3)} \log |\beta^{(2)}| - l_2^{(2)} \log |\beta^{(3)}|.
\end{equation}
By Equation \eqref{eq:lem:betaiuj-log} in Lemma \ref{lem:betaiuj} we have
\begin{align*}
	l_2^{(3)} \log |\beta^{(2)}| - l_2^{(2)} \log |\beta^{(3)}|
	&= \log |y| (l_2^{(3)}-l_2^{(2)}) + l_1^{(2)} l_2^{(3)}  - l_1^{(3)}l_2^{(2)} + (|l_2^{(3)}|+|l_2^{(2)}|) L(0.4^n)\\
	&= \log |y| (l_2^{(3)}-l_2^{(2)}) + (l_1^{(2)} l_2^{(3)}  - l_1^{(3)}l_2^{(2)}) +  L(0.5^n),
\end{align*}
where we used Lemma \ref{lem:lik} to estimate $|l_2^{(3)}|+|l_2^{(2)}| \leq 2n$ and $L(2n \cdot 0.4^n)=L(0.5^n)$.
Together with \eqref{eq:u1} this implies
\begin{equation}\label{eq:u1-fertig}
	\pm R (u_1 - 1)
	= 	\log |y| (l_2^{(3)}-l_2^{(2)}) + L(0.5^n).
\end{equation}

Similarly, solving for $u_2$ we obtain
\begin{equation*}
	(l_2^{(2)} l_1^{(3)}- l_2^{(3)}l_1^{(2)}) u_2 = l_1^{(3)} \log |\beta^{(2)}| - l_1^{(2)} \log |\beta^{(3)}|
\end{equation*}
and with Equation \eqref{eq:lem:betaiuj-log}
\begin{align}
	\mp R u_2 
	&= \log |y| (l_1^{(3)}-l_1^{(2)}) + l_1^{(3)}l_1^{(2)} - l_1^{(2)}l_1^{(3)}
	+ (|l_1^{(3)}|+|l_1^{(2)}|) L(0.4^n)\nonumber\\
	&= \log |y| (l_1^{(3)}-l_1^{(2)}) + L(0.5^n).\label{eq:u2-fertig}
\end{align}
Since neither of the factors $(l_2^{(3)}-l_2^{(2)}), (l_1^{(3)}-l_1^{(2)})$ is small, we cannot directly repeat the argument from the proof of Lemma \ref{lem:j3}. 
In order to combine Equations \eqref{eq:u1-fertig} and \eqref{eq:u2-fertig} in a helpful way, we estimate these two factors. Using Lemma \ref{lem:lik} we obtain
\begin{align*}
l_2^{(3)}-l_2^{(2)}
	&= n \log 2 + L(0.81^n) - ( -n(\log \phi + \log 2) + \log \sqrt{5} + L(3\cdot 0.81^n))\\
	&= n( 2\log 2 + \log \phi) - \log \sqrt{5} + L(4\cdot 0.81^n),\\
l_1^{(3)}-l_1^{(2)}
	&= n \log 2 + L(0.81^n) - (n\log \phi - \log \sqrt{5} + L(2 \cdot 0.81^n))\\
	&= n(\log 2 - \log \phi) + \log \sqrt{5} + L(3 \cdot 0.81^n).
\end{align*}
Thus \eqref{eq:u1-fertig} and \eqref{eq:u2-fertig} become
\begin{align}
	\pm R (u_1 - 1)
	&= \log |y| (n( 2\log 2 + \log \phi) - \log \sqrt{5} + L(4 \cdot 0.81^n) )
		+ L(0.5^n), \label{eq:u1-fertig2}\\
	\mp R u_2
	&= \log |y| (n(\log 2 - \log \phi) + \log \sqrt{5} + L(3 \cdot 0.81^n) )
		+ L(0.5^n).\label{eq:u2-fertig2}
\end{align}
Now we subtract \eqref{eq:u2-fertig2} multiplied by $(n( 2\log 2 + \log \phi) - \log \sqrt{5})$ from  \eqref{eq:u1-fertig2} multiplied by $(n(\log 2 - \log \phi) + \log \sqrt{5})$ and obtain
\begin{multline}\label{eq:cj1-kombi-davor}
	\pm R ((u_1 -1) (n(\log 2 - \log \phi) + \log \sqrt{5}) + u_2 (n( 2\log 2 + \log \phi) - \log \sqrt{5}))\\
	= \log |y| \cdot L(4\cdot 0.81^n \cdot 0.3n + 3 \cdot 0.81^n \cdot 2n) + L(0.5^n \cdot 3n).
\end{multline}
Using the assumption $n>1000$ to estimate the $L$-terms, we can rewrite the equation as
\begin{equation}\label{eq:cj1-kombi}
	\pm R ( x_1 \log 2 + x_2 \log \phi + x_3 \log \sqrt{5})
	= \log |y| \cdot L(0.82^n),
\end{equation}
where
\begin{align*}
	x_1 &= n(u_1 + 2 u_2 -1),\\
	x_2 &= n(-u_1 + u_2 + 1), \\
	x_3 &= u_1 - u_2 -1. 
\end{align*}
We set
\[
	\Lambda := x_1 \log 2 + x_2 \log \phi + x_3 \log \sqrt{5}.
\]
Since $2$, $\phi$ and $\sqrt{5}$ are multiplicatively independent, $\Lambda$ can only be zero if $x_1=x_2=x_3=0$. For $n>0$ it is easy to see that this is only the case if $u_1=1$ and $u_2=0$.
But then we have $x-\alpha y =\beta=\pm \eta_1^{u_1}\eta_2^{u_2}= \pm \eta_1 = \pm \alpha$, which means that we have the trivial solution $(x,y)=(0,\mp 1)$. Thus we may assume $\Lambda \neq 0$ and apply Matveev's bound (Lemma \ref{lem:matveev}) to $\Lambda$ with $\gamma_1=2$, $\gamma_2= \phi$, $\gamma_3=\sqrt{5}$ and $b_i=x_i$ ($i=1,2,3$). We have $D=2$ and we may set
\begin{align*}
A_1 &:= 1.4 > \max\{2 h(2),|\log 2|, 0.16\} = 2 \log 2,\\
A_2 &:=  0.5 > \max\{2 h(\phi),|\log \phi|, 0.16\} = \log \phi,\\
A_3 &:=  1.7 > \max\{2 h(\sqrt{5}),|\log \sqrt{5}|, 0.16\} = \log 5.
\end{align*}
Moreover, using Lemma \ref{lem:bound_u} we can set
\begin{align*}
	\max\{|x_1|,|x_2|,|x_3|\}
	&\leq n( |u_1| + 2|u_2| + 1) 
	< n \left( \frac{2\log|y|}{n} + 2 + 2 \left( \frac{2\log|y|}{n} +2\right) +1 \right)\\
	&= 6 \log |y| + 7n =:B. 
\end{align*}
Then Lemma \ref{lem:matveev} yields
\begin{equation}\label{eq:Lambda-lowerbound}
	\log |\Lambda|
	> - C_2 \log(e (6 \log |y| + 7n)),
\end{equation}
with $ C_2= 1.13 \cdot 10^{12} > 1.4\cdot 10^{11} \cdot 2^2 \cdot \log (e\cdot 2) \cdot 1.4 \cdot 0.5 \cdot 1.7$.

Recall that by \eqref{eq:cj1-kombi} we have
\begin{equation}\label{eq:lambda2}
	R \cdot |\Lambda| \leq \log |y| \cdot 0.82^n,
\end{equation}
which implies
\[
	\log |\Lambda|
	\leq \log \log |y| + n \log 0.82 - \log R.
\]
Combining this with \eqref{eq:Lambda-lowerbound} and noting that $\log R> 0$ by Lemma \ref{lem:reg_index} we obtain
\[
	- C_2 \log(e (6 \log |y| + 7n))
	< \log \log |y| + n \log 0.82.
\]
Since $\log(e (6 \log |y| + 7n)) <  \log \log |y| + \log n + 4$, this yields
\[
	(-\log 0.82) \cdot n - C_2 \log n - 4 C_2 
	< (C_2 + 1) \log \log |y|.
\]
This implies
\[
	\frac{n}{(C_2+1)/(-\log 0.82)} - \log n - 4
	< \log \log |y|,
\]
which is the desired lower bound.

\noindent\textbf{Case 2:} $j=2$.
This time we consider the first and third conjugate:
\begin{align*}
	\log |\beta^{(1)}| = u_1 l_1^{(1)} + u_2 l_2^{(1)},\\
	\log |\beta^{(3)}| = u_1 l_1^{(3)} + u_2 l_2^{(3)}.
\end{align*}
Solving the system of linear equations for $u_1$ and $u_2$ and using Equation \eqref{eq:lem:betaiuj-log}, we obtain analogously to Case 1:
\begin{align}
	\pm R u_1
	&= 	\log |y| (l_2^{(3)}-l_2^{(1)}) + L(0.5^n),\label{eq:u1-fertig-casej2}\\
	\mp R ( u_2 - 1)
	&= \log |y| (l_1^{(3)}-l_1^{(1)}) + L(0.5^n).\label{eq:u2-fertig-casej2}
\end{align}
We estimate the factors $(l_2^{(3)}-l_2^{(1)})$ and $(l_1^{(3)}-l_1^{(1)})$ using Lemma \ref{lem:lik}:
\begin{align*}
l_2^{(3)}-l_2^{(1)}
	&= n(\log 2 - \log \phi) + \log \sqrt{5} + L(3\cdot 0.81^n),\\
l_1^{(3)}-l_1^{(1)}
	&= n( 2\log 2 + \log \phi) - \log \sqrt{5} + L(4 \cdot 0.81^n).
\end{align*}
Thus \eqref{eq:u1-fertig-casej2} and \eqref{eq:u2-fertig-casej2} become
\begin{align*}
	\pm R u_1
	&= \log |y| (n(\log 2 - \log \phi) + \log \sqrt{5} + L(3\cdot 0.81^n) )
		+ L(0.5^n),\\
	\mp R ( u_2 - 1)
	&= \log |y| (n( 2\log 2 + \log \phi) - \log \sqrt{5} + L(4 \cdot 0.81^n) )
		+ L(0.5^n).
\end{align*}
But now we are in exactly the same situation as in Case 1, except that $u_1$ and $u_2$ have switched roles. Thus we can repeat all the arguments from above. The only detail that changes is that $\Lambda=0$ then yields $u_1=0, u_2=1$, which corresponds to the trivial solution $(x,y)=\pm(F_n, 1)$.
\end{proof}

\begin{rem}
In \cite{Heuberger2001}, the corresponding part to our Lemma \ref{lem:j12} is Section 7. In \cite[Sec. 7]{Heuberger2001} the computations are more complicated because the Thue equations are of arbitrary degree. However, the computations there do not require another application of lower bounds for linear forms in logarithms: Because the author takes logarithms of polynomials (instead of linear recurrence sequences), he gets coefficients that involve the exponents of the polynomials, i.e. integers. Thus the equation corresponding to \eqref{eq:cj1-kombi-davor} has an integer factor on the left hand side instead of some logarithms, and it is enough to check that the equation is non-zero in order to get a good lower bound.

We, on the other hand, had to apply a lower bound for linear forms in logarithms. Note that another way to obtain the linear form in logarithms is to modify the linear form from the proof of Lemma \ref{lem:upperBoundY}. Such a trick has been done in \cite{HilgartVukusicZiegler2021}.
\end{rem}

\section{Absolute bound for $n$ and finishing the proof}\label{sec:finish}

Comparing the bounds in Lemma \ref{lem:upperBoundY}, Lemma \ref{lem:j3} and Lemma \ref{lem:j12}, we immediately obtain an absolute bound for $n$:

\begin{lem}\label{lem:absoluteBound}
There are no non-trivial solutions with $n>1000$ and $j=3$.
If $j\in \{1,2\}$, then
\[
	n < 1.03 \cdot 10^{15}.
\]
\end{lem}

In order to finish the proof of Theorem \ref{thm:main} we only need to deal with the cases $j=1,2$ and $1000< n < 1.03 \cdot 10^{15}$.

\begin{lem}\label{lem:j12-reduktion}
There are no non-trivial solutions with $1000< n < 1.03 \cdot 10^{15}$ and $j\in \{1,2\}$.
\end{lem}
\begin{proof}
We go back to Inequality \eqref{eq:lambda2} from the proof of Lemma \ref{lem:j12}, which together with Lemma \ref{lem:reg_index} and Lemma \ref{lem:upperBoundY} implies
\begin{align*}
	|\Lambda|
	=|x_1 \log 2 + x_2 \log \phi + x_3 \log \sqrt{5}|
	&\leq \log|y| \cdot 0.82^n \cdot R^{-1}\\
	&< 3.66 \cdot 10^{16} \cdot n^3 \cdot 0.82^n \cdot n^{-2}
	= 3.66 \cdot 10^{16} \cdot 0.82^n \cdot n,
\end{align*}
where the $x_i$ are integers bounded in abolute values by
\[
	6 \log |y| + 7n 
	< 6 \cdot 3.66 \cdot 10^{16} \cdot n^3 + 7n
	< 2.2 \cdot 10^{17} \cdot n^3
	:=M(n).
\]
Then we apply the LLL-algorithm to find an absolute lower bound for $|\Lambda|$ as described in Lemma \ref{lem:LLL} with $M=M(1.03 \cdot 10^{15})$. To obtain the matrices $B$ and $B^*$ we use the matrix attributes \verb|LLL()| and \verb|gram_schmidt()| of Sage \cite{sagemath}.
Indeed, with $C=10^{192}$ the procedure works and we obtain 
\[
	6.37 \cdot 10^{-127}
	< |\Lambda|
	< 3.66 \cdot 10^{16} \cdot 0.82^n \cdot n,
\]
which implies $n< 1694$.

We repeat the whole procedure once more, this time using $M(1694)$ and $C=10^{83}$. Then we obtain the inequality
\[
	3.67 \cdot 10^{-56}
	< |\Lambda|
	< 3.66 \cdot 10^{16} \cdot 0.82^n \cdot n,
\]
which implies $n<870$ and we are done.
\end{proof}

\begin{rem}
In the course of this paper we did a lot of estimations, many of which were not very sharp but also not relevant. For instance, the term $L(0.5^n)$ in Lemma \ref{lem:roots-estimate} could have been chosen much smaller, but that would have had no effect. The estimate that probably played the most important role was $\phi/2 \approx 0.81$. This played a role in Lemma \ref{lem:j3} (it resulted in the constant $1.2$) as well as in the proofs of Lemma \ref{lem:j12} and Lemma~\ref{lem:j12-reduktion} (where it finally showed as the constant $0.82$). This suggests that in the generalised problem
\[
	X(X-G_1(n) Y)(X-G_2(n) Y) - Y^3 = \pm 1,
\]
where $G_1(n), G_2(n)$ are any linear recurrence sequences,  
the ratio between the dominant roots of $G_1,G_2$ might be essential for the approach used in this paper.
\end{rem}

\section*{Acknowledgement}
The author wants to thank Volker Ziegler for his encouragement and many helpful discussions and suggestions.

\bibliographystyle{plain}
\bibliography{lit_thom}

\begin{thebibliography}{10}

\bibitem{Baker1984}
A.~Baker.
\newblock {\em A Concise Introduction to the Theory of Numbers}.
\newblock Cambridge University Press, 1984.

\bibitem{Bennett2001}
M.~A. Bennett.
\newblock Rational approximation to algebraic numbers of small height: the
  {D}iophantine equation {$|ax^n-by^n|=1$}.
\newblock {\em J. Reine Angew. Math.}, 535:1--49, 2001.

\bibitem{BernsteinHasse1969}
L.~Bernstein and H.~Hasse.
\newblock An explicit formula for the units of an algebraic number field of
  degree {$n\geq 2$}.
\newblock {\em Pacific J. Math.}, 30:293--365, 1969.

\bibitem{BugeaudGyory1996}
Y.~Bugeaud and K.~Gy\H{o}ry.
\newblock Bounds for the solutions of {T}hue-{M}ahler equations and norm form
  equations.
\newblock {\em Acta Arith.}, 74(3):273--292, 1996.

\bibitem{DujellaPetho1998}
A.~Dujella and A.~Peth\H{o}.
\newblock Generalization of a theorem of {B}aker and {D}avenport.
\newblock {\em Quart. J. Math.}, 49(195):291--306, 1998.

\bibitem{Heuberger2001}
C.~Heuberger.
\newblock On a conjecture of {E}. {T}homas concerning parametrized {T}hue
  equations.
\newblock {\em Acta Arith.}, 98(4):375--394, 2001.

\bibitem{Heuberger2006}
C.~Heuberger.
\newblock Parametrized thue equations -- a survey.
\newblock In {\em Proceedings of the RIMS symposium “Analytic Number Theory
  and Surrounding Areas”}, volume 1511 of {\em RIMS K{\^o}ky{\^u}roku}, pages
  82--91, 2006.

\bibitem{HilgartVukusicZiegler2021}
T.~Hilgart, I.~Vukusic, and V.~Ziegler.
\newblock On a family of cubic {T}hue {E}quations involving {F}ibonacci and
  {L}ucas numbers, 2021.
\newblock Preprint, \href{https://arxiv.org/abs/2106.03509 }{arXiv:2106.03509}.

\bibitem{Matveev2000}
E.~M. Matveev.
\newblock An explicit lower bound for a homogeneous rational linear form in the
  logarithms of algebraic numbers. {II}.
\newblock {\em Izv. Math.}, 64(6):1217--1269, 2000.

\bibitem{Smart1998}
N.~P. Smart.
\newblock {\em {The Algorithmic Resolution of Diophantine Equations}}.
\newblock London Mathematical Society Studen Texts 41. Cambridge University
  Press, 1998.

\bibitem{Stender1972}
H.-J. Stender.
\newblock Einheiten f\"{u}r eine allgemeine {K}lasse total reeller
  algebraischer {Z}ahlk\"{o}rper.
\newblock {\em J. Reine Angew. Math.}, 257:151--178, 1972.

\bibitem{PARI}
{The PARI~Group}, Bordeaux.
\newblock {\em {PARI/GP, version {\tt 2.1.5}}}, 2004.
\newblock
  \href{http://pari.math.u-bordeaux.fr/}{http://pari.math.u-bordeaux.fr/}.

\bibitem{sagemath}
{The Sage Developers}.
\newblock {\em {S}ageMath, the {S}age {M}athematics {S}oftware {S}ystem
  ({V}ersion 9.0)}, 2020.
\newblock \href{https://www.sagemath.org}{https://www.sagemath.org}.

\bibitem{Thomas1979}
E.~Thomas.
\newblock Fundamental units for orders in certain cubic number fields.
\newblock {\em J. Reine Angew. Math.}, 310:33--55, 1979.

\bibitem{Thomas1990}
E.~Thomas.
\newblock Complete solutions to a family of cubic {D}iophantine equations.
\newblock {\em J. Number Theory}, 34(2):235--250, 1990.

\bibitem{Thomas1993}
E.~Thomas.
\newblock Solutions to certain families of {T}hue equations.
\newblock {\em J. Number Theory}, 43(3):319--369, 1993.

\bibitem{Thue1909}
A.~Thue.
\newblock \"{U}ber {A}nn\"{a}herungswerte algebraischer {Z}ahlen.
\newblock {\em J. Reine Angew. Math.}, 135:284--305, 1909.

\bibitem{Thue1918}
A.~Thue.
\newblock Berechnung aller {L}ösungen gewisser {G}leichungen von der {F}orm
  $ax^r-by^r=f$.
\newblock In T.~Nagell, Selberg A., S.~Selberg, and Thalberg K., editors, {\em
  Selected {M}athematical {P}apers of {A}xel {T}hue}, pages 565--571.
  Universitetsforlaget, Oslo, Bergen, Troms{\o}, 1977.
\newblock Originally publ.: Kra. Vidensk. Selsk. Skrifter I Mat. Nat. K1. 1918.
  No. 4. Kra. 1919.

\bibitem{Ziegler2007}
V.~Ziegler.
\newblock Thomas' conjecture over function fields.
\newblock {\em J. Th\'{e}or. Nombres Bordeaux}, 19(1):289--309, 2007.

\end{thebibliography}

\end{document}